\documentclass[11pt]{amsart}

\usepackage{hyperref}
\usepackage[usenames]{color}
\usepackage{amsmath,amsthm,amsfonts,amssymb}

\newtheorem{theorem}{Theorem}[section]
\newtheorem{lemma}[theorem]{Lemma}

\theoremstyle{definition}
\newtheorem{example}[theorem]{Example}

\theoremstyle{remark}
\newtheorem{remark}[theorem]{Remark}

\numberwithin{equation}{section}

\begin{document}

\title[$p$-solvability of regular equations over unitriangular groups]{$p$-solvability of regular equations over unitriangular groups over prime finite fields}

\author{Anton Menshov}
\address{Institute of Mathematics and Information Technologies\\Omsk State Dostoevskii University}
\curraddr{}
\email{menshov.a.v@gmail.com}
\thanks{}

\author{Vitali\u\i\ Roman'kov}
\address{Institute of Mathematics and Information Technologies\\Omsk State Dostoevskii University}
\curraddr{}
\email{romankov48@mail.ru}
\thanks{}


\date{}

\begin{abstract}
An equation over a group with one unknown is called \emph{regular} if the exponent sum of the unknown is nonzero.
In this paper we prove that some regular equations of exponent $rp^s$, where $r \in \mathbb{Z}$, $s \in \mathbb{N}$, $\gcd(r,p)=1$, over the group UT$_n(\mathbb{F}_p)$ ($n \geq 2$) are solvable in an overgroup isomorphic to UT$_{(n-1)p^s + 1}(\mathbb{F}_p)$.
Applying this for $n=3$ we prove that any regular equation of exponent $rp^s$ over the Heisenberg $p$-group UT$_3(\mathbb{F}_p)$ is solvable in an overgroup isomorphic to UT$_{2p^s + 1}(\mathbb{F}_p)$.
The proofs of these results are constructive and allow to obtain solutions of equations in explicit form.

\end{abstract}

\maketitle

\section{Introduction}
\label{sec:intro}

An equation with the variable $x$ over a group $G$ is an expression of the form
\begin{equation}\label{eq:one_var_eq}
u(x)=1,
\end{equation}
where
\begin{equation}\label{eq:one_var_eq_left_part}
u(x) = x^{\epsilon_1} g_1 x^{\epsilon_2} g_2 \dots x^{\epsilon_n} g_n \in G \ast \langle x \rangle.
\end{equation}

If $H$ is a bigger group, i.e., a group containing $G$ as a fixed subgroup, then an equation over $G$ could be also considered as an equation over $H.$
Equation~(\ref{eq:one_var_eq}) is {\it solvable in $G$} if there is an element $g \in G$ such that $w(g)=1$.
Equation~(\ref{eq:one_var_eq}) is {\it solvable over $G$} if there is an overgroup $H \geq G$ where this equation has a solution.

Let $\mathcal{C}$ be a class of groups.
An equation over a group $G \in \mathcal{C}$ is called {\it solvable in the class $\mathcal{C}$} if there is a group $H \in \mathcal{C}$, containing $G$, where this equation has a solution.

Equation~(\ref{eq:one_var_eq}) is called {\it regular} if its {\it exponent} $\epsilon = \sum_{i=1}^n \epsilon_i$ is nonzero.
According to the famous {\it Kervaire-Laudenbach conjecture} (see~\cite{romankov}) every regular equation is solvable over an arbitrary group.
In~\cite{romankov} similar conjectures were proposed for the classes of nilpotent and solvable groups.
The principal question for the nilpotent version of this conjecture is solvability of regular equations in the class $\mathcal{F}_p$ of all finite $p$-groups, for all prime $p$.
Observe that, in general, Kervaire-Laudenbach conjecture (and its generalization for systems of equations) is still opened.
For existing partial result we refer to the survey~\cite{romankov}.

It follows by the results of M.~Gerstenhaber and O.~Rothaus~\cite{gerstenhaber_rothaus} that any regular equation over a finite group is solvable in a finite overgroup.
We remark that existing proofs of these results are not constructive.
They don't allow to construct an overgroup where the equation is solvable.
Also they don't allow to describe the structure of an overgroup.
Other constructive proofs for these results are still not found.
Unfortunately nothing could be said about solvability of equations in finite overgroups from some class of groups.

An equation over a finite $p$-group $G$ is called {\it $p$-solvable} if there is a finite $p$-overgroup $H \geq G$ where it has a solution.

Let $\mathbb{F}_p$ be the prime finite field of order $p$ and UT$_n(\mathbb{F}_p)$ ($n \geq 2$) be the group of $n \times n$ upper unitriangular matrices over  $\mathbb{F}_p$.
It is well known that any finite $p$-group $G$ is isomorphic to a subgroup of UT$_{|G|}(\mathbb{F}_p)$.
Hence to prove solvability of regular equations in the class $\mathcal{F}_p$ of all finite $p$-groups it is enough to show that any regular equation over UT$_n(\mathbb{F}_p)$, for $n \geq 2$, is $p$-solvable.
Observe that the case $n=2$ corresponds to adjunction of roots to cyclic groups of order $p$ in the class $\mathcal{F}_p$ and thus is trivial.
So, the first nontrivial case corresponds to $n=3$.

In~\cite{menshov_romankov_reg_eq} the authors proved $p$-solvability of some regular equations over the Heisenberg $p$-group UT$_3(\mathbb{F}_p)$.
In this paper we extend the results of~\cite{menshov_romankov_reg_eq}.
In Theorem~\ref{th:main} we prove that some regular equations of exponent $rp^s$, where $r \in \mathbb{Z}$, $s \in \mathbb{N}$, $\gcd(r,p)=1$, over the group UT$_n(\mathbb{F}_p)$ ($n \geq 2$) are solvable in an overgroup isomorphic to UT$_{(n-1)p^s+1}(\mathbb{F}_p)$.
Applying this result for $n=3$ we prove in Theorem~\ref{th:UT_3} that any regular equation of exponent $rp^s$ over the Heisenberg $p$-group UT$_3(\mathbb{F}_p)$ is solvable in an overgroup isomorphic to UT$_{2p^s+1}(\mathbb{F}_p)$.
The proofs of these results are constructive and allow to obtain solutions of equations in explicit forms.

\section{Preliminaries}
\label{sec:preliminaries}

Let $\Delta = \{ \delta_1, \delta_2, \dots, \delta_n \}$ be rational numbers where $\delta_i < \delta_{i+1}$, for $i=1,\dots,n-1$.
Write $\mathfrak{A}$ for the associative algebra over a field $\mathbb{F}$ with basis $\{ e_{\delta_i,\delta_{i+1}} \mid i=1,\dots,n-1 \}$, such that the multiplication in $\mathfrak{A}$ is defined by

\[
e_{\alpha,\beta} \cdot e_{\beta,\gamma} = e_{\alpha,\gamma}, \quad
e_{\alpha,\beta} \cdot e_{\gamma,\delta} = 0 \textrm{ for } \beta \neq \gamma .
\]
Observe that  $\mathfrak{A}$ is isomorphic to the algebra NT$_n(\mathbb{F})$ of all $n \times n$ upper nil-triangular matrices over  $\mathbb{F}$.
For an element
\begin{equation}
\label{eq:algebra_element_form}
u = \sum_{ \substack{ \alpha , \beta \in \Delta, \\ \alpha < \beta }} u_{ \alpha , \beta } e_{ \alpha , \beta } \in \mathfrak{A}, \quad u_{ \alpha , \beta } \in \mathbb{F},
\end{equation}
denote
\[
\mathrm{supp}(u) = \{ (\alpha,\beta) \in \Delta^2 \mid u_{\alpha,\beta} \neq 0 \}.
\]

\noindent
With  $u$ we associate a weighted oriented graph $\Gamma(u)$ with vertices corresponding to rational numbers from $\Delta$.
Vertices $\alpha,\beta \in \Delta$ are connected with the edge $(\alpha,\beta)$ if and only if $(\alpha,\beta) \in$ supp$(u)$.
We set $u_{\alpha,\beta}$ from (\ref{eq:algebra_element_form}) as the weight of  $(\alpha,\beta).$
Let $\mathcal{P}$ be a $\gamma\delta$-path ($\gamma < \delta$) in $\Gamma (u)$, i.e., a path from $\gamma$ to $\delta$.
The weight $w(\mathcal{P})$ of $\mathcal{P}$ is defined as the product of the weights of all edges forming the path.
The length $|\mathcal{P}|$ is the number of its edges.
We will say that  element $u \in \mathfrak{A}$ has the length $l = l(u)$ if the maximal length of a path in $\Gamma(u)$ is equal to $l$.
It is clear that $u^{l+1} = 0.$
The length $l(\mathfrak{M})$ of an arbitrary subset $\mathfrak{M} \subseteq \mathfrak{A}$ is defined by $\sup \{ l(u) \mid u \in \mathfrak{M} \}$.
Observe that $l(\mathfrak{A})=n-1$.

Add to $\mathfrak{A}$ the external unit $1$.
Then the set G$(\mathfrak{A}) = \{ 1+u \mid u \in \mathfrak{A} \}$ is a group since
\begin{align*}
(1+u)(1+v) &= 1 + (u+v+uv), \\
(1+u)^{-1} &= 1 + \sum_{i=1}^{n-1} (-1)^i u^i.
\end{align*}
By $t_{\alpha,\beta}$ ($\alpha,\beta \in \Delta, \; \alpha < \beta$) we denote a {\it transvection} $1 + e_{\alpha,\beta}$.
Also for any $\gamma \in \mathbb{F}$ we denote $t_{\alpha,\beta }(\gamma) = 1 + \gamma e_{\alpha,\beta}.$
Elements of the form $e_{\alpha,\beta}$ will be called {\it matrix units}.
Denote by $\mathfrak{A}^{(i)}$ the set of all sums of products of $i$ elements of $\mathfrak{A}$, for $i > 0$.
Clearly $\mathfrak{A}^{(i)}$ is a subalgebra of $\mathfrak{A}$ and G$(\mathfrak{A}^{(i)}) = \{ 1 + u \mid u \in \mathfrak{A}^{(i)} \}$ is a subgroup of G$(\mathfrak{A})$.

\begin{lemma}
\label{lm:algebra_group_relation}
The following statements hold:
\mbox{}
\begin{enumerate}
\item[1)]  $\mathfrak{A}^{(n)}=0,$ hence $\mathfrak{A}$ is nilpotent of class $n.$
\item[2)] The series $1 = G(\mathfrak{A}^{(n)}) \leq G(\mathfrak{A}^{(n-1)}) \leq \dots \leq G(\mathfrak{A}^{(1)}) = G(\mathfrak{A})$ is the lower (and upper) central series of G$(\mathfrak{A})$.
\item[3)] The group G$(\mathfrak{A})$ is nilpotent of class $n-1.$
\end{enumerate}
\end{lemma}

Observe that G$(\mathfrak{A})$ is isomorphic to the group UT$_n(\mathbb{F})$ of $n \times n$ upper unitriangular matrices over a field $\mathbb{F}$.
The group G$(\mathfrak{A}^{(i)})$ consists of matrices whose first $i-1$ superdiagonals are zero.

Further we will take $\mathbb{F}$ to be the prime finite field $\mathbb{F}_p$ of order $p$.

In~\cite{menshov_romankov} for $q=p^s$, $s \in \mathbb{Z}^+$, $n \geq 2$, the authors constructed embeddings of UT$_n(\mathbb{F}_p)$ in UT$_m(\mathbb{F}_p)$, where $m=(n-1)q+1$, such that any element of UT$_n(\mathbb{F}_p)$ has a $q$-th root in UT$_m(\mathbb{F}_p)$.
Since these embeddings play essential role further, we will describe them in details.

Let $\alpha_{i,j} \in \mathbb{Q}$ be such that
\[
i < \alpha_{i,1} < \dots < \alpha_{i,q-1} < i+1, \qquad i=1,\dots,n-1,
\]
and let UT$_m(\mathbb{F}_p)$ be generated by
\[
t_{i, \alpha_{i,1}}', t_{\alpha_{i,1}, \alpha_{i,2}}', \dots, t_{\alpha_{i,q-1}, i+1}', \qquad i=1,\dots,n-1.
\]
Consider the embedding $\phi:$ UT$_n(\mathbb{F}_p) \to$ UT$_m(\mathbb{F}_p)$ defined by
\begin{equation}
\label{eq:UT_embedding_1}
\phi: \begin{array}{rcl}
    t_{1,2} & \mapsto & t_{1,2}', \\
    t_{2,3} & \mapsto & t_{2,3}' t_{\alpha_{1,1}, \alpha_{2,1}}' t_{\alpha_{1,2}, \alpha_{2,2}}' \dots t_{\alpha_{1,q-1}, \alpha_{2,q-1}}', \\
    t_{3,4} & \mapsto & t_{3,4}' t_{\alpha_{2,1}, \alpha_{3,1}}' t_{\alpha_{2,2}, \alpha_{3,2}}' \dots t_{\alpha_{2,q-1}, \alpha_{3,q-1}}', \\
    & \dots & \\
    t_{n-1,n} & \mapsto & t_{n-1,n}' t_{\alpha_{n-2,1}, \alpha_{n-1,1}}' t_{\alpha_{n-2,2}, \alpha_{n-1,2}}' \dots t_{\alpha_{n-2,q-1}, \alpha_{n-1,q-1}}'.
\end{array}
\end{equation}

\begin{example}
Let $n=p=q=3$, then the image of
\[
a = \left( \begin{smallmatrix}
1 & a_{12} & a_{13} \\
0 & 1      & a_{23} \\
0 & 0      & 1
\end{smallmatrix} \right) \in \mathrm{UT}_3(\mathbb{F}_3)
\]
under embedding~(\ref{eq:UT_embedding_1}) is equal to
\[
\phi(a) = \left( \begin{smallmatrix}
1 & 0 & 0 & a_{12} & 0 & 0 & a_{13} \\
0 & 1 & 0 & 0 & a_{23} & 0 & 0 \\
0 & 0 & 1 & 0 & 0 & a_{23} & 0 \\
0 & 0 & 0 & 1 & 0 & 0 & a_{23} \\
0 & 0 & 0 & 0 & 1 & 0 & 0 \\
0 & 0 & 0 & 0 & 0 & 1 & 0 \\
0 & 0 & 0 & 0 & 0 & 0 & 1
\end{smallmatrix} \right) \in \mathrm{UT}_7(\mathbb{F}_3)
\]
\end{example}

Similarly one can define the embedding $\psi:$ UT$_n(\mathbb{F}_p) \to$ UT$_m(\mathbb{F}_p)$ by
\begin{equation}
\label{eq:UT_embedding_2}
\psi: \begin{array}{rcl}

    t_{1,2} & \mapsto & t_{1,2}' t_{\alpha_{1,1}, \alpha_{2,1}}' t_{\alpha_{1,2}, \alpha_{2,2}}' \dots t_{\alpha_{1,q-1}, \alpha_{2,q-1}}', \\
    t_{2,3} & \mapsto & t_{2,3}' t_{\alpha_{2,1}, \alpha_{3,1}}' t_{\alpha_{2,2}, \alpha_{3,2}}' \dots t_{\alpha_{2,q-1}, \alpha_{3,q-1}}', \\
    & \dots & \\
    t_{n-2,n-1} & \mapsto & t_{n-2,n-1}' t_{\alpha_{n-2,1}, \alpha_{n-1,1}}' t_{\alpha_{n-2,2}, \alpha_{n-1,2}}' \dots t_{\alpha_{n-2,q-1}, \alpha_{n-1,q-1}}', \\
    t_{n-1,n} & \mapsto & t_{n-1,n}'.
\end{array}
\end{equation}

\begin{example}
Let $n=p=q=3$, then the image of
\[
a = \left( \begin{smallmatrix}
1 & a_{12} & a_{13} \\
0 & 1      & a_{23} \\
0 & 0      & 1
\end{smallmatrix} \right) \in \mathrm{UT}_3(\mathbb{F}_3)
\]
under embedding~(\ref{eq:UT_embedding_2}) is equal to
\[
\psi(a) = \left( \begin{smallmatrix}
1 & 0 & 0 & a_{12} & 0 & 0 & a_{13} \\
0 & 1 & 0 & 0 & a_{12} & 0 & 0 \\
0 & 0 & 1 & 0 & 0 & a_{12} & 0 \\
0 & 0 & 0 & 1 & 0 & 0 & a_{23} \\
0 & 0 & 0 & 0 & 1 & 0 & 0 \\
0 & 0 & 0 & 0 & 0 & 1 & 0 \\
0 & 0 & 0 & 0 & 0 & 0 & 1
\end{smallmatrix} \right) \in \mathrm{UT}_7(\mathbb{F}_3)
\]
\end{example}

\section{Main results}
\label{sec:main}

Consider an equation over a group $G$ of the form
\[
\underbrace{ x^{\epsilon_1} g_1 x^{\epsilon_2} g_2 \dots x^{\epsilon_n} g_n }_{u(x)} = 1
\]
with exponent $\epsilon$.
Without the loss of generality we suppose that $\epsilon < 0$.
Using the identity $fg=gf[f,g]$ and moving the powers of unknown and coefficients of equation to the left we obtain an equation of the form
\begin{equation}\label{eq:one_var_eq_modified}
x^{-\epsilon} = u(1) v(x),
\end{equation}
where $v(x) = \prod\limits_{i=1}^l C_i(x)$ ($l \geq 0$) and
\begin{equation}\label{eq:commutator_form}
C_i(x) = [u_i, x^{\sigma_i}, S_{i,1}(x), \dots, S_{i,k_i}(x)]
\end{equation}
is a commutator of weight $k_i + 2$, for $k_i \geq 0$, $u_i \in G$, $\sigma_i = \pm 1$, $S_{i,j}(x)=g_{ij} \in G$ or $S_{i,j}(x)=x^{\tau_{ij}}$, $\tau_{ij}=\pm 1$.
Observe that $u(1)=g_1 g_2 \dots g_n$ and exponent of $v(x)$ is equal to $0$.

\begin{example}
We will provide the sequence of transformations reducing the equation $xg_1xg_2=1$ to the form~(\ref{eq:one_var_eq_modified}):
\begin{align*}
1 &= xg_1xg_2, \\
1 &= xxg_1[g_1,x]g_2, \\
1 &= xxg_1g_2[g_1,x][g_1,x,g_2], \\
x^{-2} &= g_1g_2[g_1,x][g_1,x,g_2], \\
x^2 &= g_1g_2[g_1,x^{-1}][g_1,x^{-1},g_2].
\end{align*}

\end{example}

We will prove several technical lemmas first.

\begin{lemma}
\label{lm:power}
Let $\delta_1 < \dots < \delta_n$ be rational numbers and $\mathfrak{A}$ be the associative algebra over the field $\mathbb{F}_p$, generated by $e_{\delta_i,\delta_{i+1}},$ for $i=1,\dots,n-1$.
Let $x=1+u \in$ G$(\mathfrak{A})$ and $s \in \mathbb{Z}^+$ then
\[
x^{p^s} = 1 + u^{p^s} = 1 + \sum_{ \substack{ \mathcal{P} - \alpha\beta-path, \\ |\mathcal{P}|=p^s }} w(\mathcal{P}) e_{\alpha,\beta},
\]
where the sum is over all paths $\mathcal{P}$ of length $p^s$ in $\Gamma(u)$.
\end{lemma}

\begin{remark}
\label{re:power}
In notations of Lemma~\ref{lm:power} let us add to $x$ the element $w e_{\gamma,\delta} \in \mathfrak{A}$, $w \in \mathbb{F}_p$.
It will correspond to addition of the edge $(\gamma,\delta)$ with weight $w$ between the vertices $\gamma$ and $\delta$ in $\Gamma(u)$.
Clearly all previous paths of length $p^s$ will remain in $\Gamma(u)$ and additional paths passing through $(\gamma,\delta)$ may appear.
\end{remark}

\begin{lemma}
\label{lm:commutator}
Let $\delta_1 < \dots < \delta_n$ be rational numbers and $\mathfrak{A}$ be the associative algebra over the field $\mathbb{F}_p$, generated by $e_{\delta_i,\delta_{i+1}},$ for $i=1,\dots,n-1$.
Let
\[
C(x) = [g,x^{\sigma},S_1(x),\dots,S_k(x)]
\]
be a commutator of the form~(\ref{eq:commutator_form}) over G$(\mathfrak{A})$, where $g \in G(\mathfrak{A}^{(r)})$.
If $d \in \mathfrak{A}^{(t)}$ then
\[
C(x+d) = C(x) + \mathcal{E},
\]
where $\mathcal{E} \in \mathfrak{A}^{(r+t)}$.
\end{lemma}

\begin{proof}
We will split the proof into several steps.

{\bf Step 1.}
Take $1+a \in$ G$(\mathfrak{A}^{(s)})$ and $1+b \in$ G$(\mathfrak{A}^{(w)})$ then
\begin{align*}
[1+a,1+b]
    &= ((1+b)(1+a))^{-1} (1+a)(1+b) \\
    &= (1+ \underbrace{a+b+ba}_v)^{-1} (1+ \underbrace{a+b+ab}_u)\\
    &= (1 - v + v^2 - \dots + (-1)^{n-1}v^{n-1}) (1+u) \\
    &= 1 + (1 - v + v^2 - \dots + (-1)^{n-2}v^{n-2})(u-v).
\end{align*}
Since $u-v = ab - ba \in \mathfrak{A}^{(s+w)}$ then $[1+a,1+b] \in G(\mathfrak{A}^{(s+w)})$.
Hence $[g,x^{\sigma}] \in$ G$(\mathfrak{A}^{(r+1)})$.

{\bf Step 2.}
We will prove that $[g,(x+d)^{\sigma}]=[g,x^{\sigma}]+\mathcal{E}$ where $\mathcal{E} \in \mathfrak{A}^{(r+t)}$.
Take $\sigma=1$, $g=1+a$, $x=1+b$, and $d \in \mathfrak{A}^{(t)}$ then
\begin{align*}
[1+a,1+b+d] &= ((1+b+d)(1+a))^{-1} (1+a)(1+b+d) \\
    &= (1 + \underbrace{a+b+ba}_v + \underbrace{d+da}_{v'})^{-1} (1 + \underbrace{a+b+ab}_u + \underbrace{d+ad}_{u'}) \\
    &= 1 + \left( 1 + \sum_{i=1}^{n-2} (-1)^i (v+v')^i \right) ((u+u')-(v+v')).
\end{align*}
Observe that
$
\left( 1 + \sum\limits_{i=1}^{n-2} (-1)^i (v+v')^i \right) = \left( 1 + \sum\limits_{i=1}^{n-2} (-1)^i v^i \right) + f(v,v'),
$
where $f(v,v')$ is the sum of products formed by $v$ and $v'$ such that each product contains $v'$.
Then
\begin{align}\label{eq:long}
&(1+v+v')^{-1}(1+u+u') \\ \nonumber
    &= 1 + \left( \left( 1 + \sum_{i=1}^{n-2} (-1)^i v^i \right) + f(v,v') \right) ((u-v) + (u'-v')) \\ \nonumber
    &= 1 + \left( 1 + \sum_{i=1}^{n-2} (-1)^i v^i \right) (u-v) \\ \nonumber
        &\qquad + f(v,v')(u-v) \\ \nonumber
        &\qquad + \left( \left( 1 + \sum_{i=1}^{n-2} (-1)^i v^i \right) + f(v,v') \right) (u'-v') \\ \nonumber
    &= (1+v)^{-1}(1+u) + \underbrace{f(v,v')(u-v) + \left( 1 + \sum_{i=1}^{n-2} (-1)^i (v+v')^i \right) (u'-v')}_{\mathcal{E}}.
\end{align}
Observe that $v',f(v,v') \in \mathfrak{A}^{(t)}$, $f(v,v')(u-v) \in \mathfrak{A}^{(r+t)}$ and $u'-v'=ad-da \in \mathfrak{A}^{(r+t)}$, hence $\mathcal{E} \in \mathfrak{A}^{(r+t)}$.
The case $\sigma=-1$ is similar since $(x+d)^{-1}=x^{-1}+d'$, for $d' \in \mathfrak{A}^{(t)}$.
This proves the lemma for commutators of weight $2$.

{\bf Step 3.}
Suppose that the lemma holds for commutators $C(x)$ of weight less or equal to $k$, for $k \geq 2$, i.e.,
\[
C(x+d)=C(x)+\mathcal{E},
\]
where $d \in \mathfrak{A}^{(t)}$, $\mathcal{E} \in \mathfrak{A}^{(r+t)}$, $C(x)=1+c \in$ G$(\mathfrak{A}^{(r+1)})$.
We will prove that the lemma holds for commutators of weight $k+1$.

If $f \in \mathfrak{A}$ then (similar to step $2$)
\[
[C(x)+\mathcal{E},1+f]=[C(x),1+f]+\mathcal{E}',
\]
where $\mathcal{E}' \in \mathfrak{A}^{(r+t+1)}$.

Further consider the commutator $[C(x)+\mathcal{E},(x+d)^{\sigma}]$ for $\sigma=1$ (the case $\sigma=-1$ is similar).
Take $x=1+b$ then
\begin{align*}
&[1+c+\mathcal{E},1+b+d] \\
    &= ((1+b+d)(1+c+\mathcal{E}))^{-1}(1+c+\mathcal{E})(1+b+d) \\
    &= (1+\underbrace{b+c+bc}_{v}+\underbrace{d+\mathcal{E}+b\mathcal{E}+dc+d\mathcal{E}}_{v'})^{-1} \\
        &\quad\;\, (1+\underbrace{b+c+cb}_{u}+\underbrace{d+\mathcal{E}+cd+\mathcal{E}b+\mathcal{E}d}_{u'}).
\end{align*}
Observe that $u-v=cb-bc \in \mathfrak{A}^{(r+2)}$, $v' \in \mathfrak{A}^{(t)}$ and $u'-v'=(cd-dc)+(\mathcal{E}b-b\mathcal{E})+(\mathcal{E}d-d\mathcal{E}) \in \mathfrak{A}^{(r+t)}$.
Then from~(\ref{eq:long}) if follows that
\[
[C(x)+\mathcal{E},x+d] = [C(x),x] + \mathcal{E}',
\]
where $\mathcal{E}' \in \mathfrak{A}^{(r+t)}$.
This proves the lemma for commutators of weight \linebreak $k+1$.
\end{proof}

The following lemma, in some cases, allows us to consider equations only of prime exponent $p^s$, $s \in \mathbb{Z}^+$, over a finite $p$-group.

\begin{lemma}
\label{lm:prime_exponent}
If any regular equation of exponent $p^s$, $s \in \mathbb{Z}^+$, over a finite $p$-group $G$ is solvable in a finite $p$-overgroup $H$ such that $|H| \leq f_G(p^s)$, where $f_G$ is a function of exponent $p^s$, depending on a group $G$, then any regular equation over $G$ is $p$-solvable.
\end{lemma}

\begin{proof}
Suppose that the exponent of the equation $u(x)=1$ over a group $G$ has the form $\epsilon=rp^s$, where $r \in \mathbb{Z}$, $s \in \mathbb{N}$, $\gcd(r,p)=1$.
If $s=0$ then, according to the theorem of A.~Shmel'kin~\cite{shmelkin}, this equation has a solution in $G$.
Furthe we assume that $s>0$.
Write the equation in the form~(\ref{eq:one_var_eq_modified}):
\[
x^{rp^s}=u(1)v(x).
\]
Take $t \in \mathbb{Z}^+$ such that $f_G(p^s) < p^t$.
Let $1=p^t l + r k$, for $l,k \in \mathbb{Z}$.
Making the substitution $y=x^r$, $x=y^k$ reduces the equation to the new one
\[
y^{p^s}=u(1)v(y^k)
\]
with exponent $p^s$.
Let $h \in H$ be a solution of the equation above, then $h^{p^t}=1$ and
\[
(h^k)^{p^sr} = (h^{rk})^{p^s} = (h^{p^t l} h^{rk})^{p^s} = h^{p^s} = u(1) v(h^k).
\]
Hence $h^k$ is a solution of the original equation in $H$.
\end{proof}

Suppose that an equation $u(x)=1$ over a group $G$ is given and $H$ is a homomorphic image of $G$.
{\it Replica} of this equation in $H$ is an equation over $H$ obtained from original by replacing all of its coefficients by its homomorphic images.

The following theorem allows us to solve a series of regular equations over unitriangular groups in some bigger unitriangular groups.

\begin{theorem}
\label{th:main}
Consider the regular equation $u(x)=1$ over the group UT$_n(\mathbb{F}_p)$ $(n \geq 2)$ with exponent $\epsilon=rp^s$, where $r \in \mathbb{Z}$, $s \in \mathbb{N}$, $\gcd(r,p)=1$.
Denote $u(1)=(a_{i,j})$.
If one of the following conditions hold
\begin{enumerate}
\item[1)] $a_{i,i+1} \neq 0$ for $i=2,\dots,n-1$,
\item[2)] $a_{i,i+1} \neq 0$ for $i=1,\dots,n-2$,
\item[3)] $u(1)$ is a central element in UT$_n(\mathbb{F}_p)$,
\end{enumerate}
then the equation in solvable in an overgroup isomorphic to UT$_m(\mathbb{F}_p)$, for $m=(n-1)p^s+1$.
\end{theorem}

\begin{proof}
If $s=0$ then, as it has been already mentioned, according to the theorem of A.~Shmel'kin~\cite{shmelkin}, the equation has a solution in UT$_n(\mathbb{F}_p)$.
Further we assume that $s>0$.
Denote, for brevity, $q=p^s$ and consider the case when $r=1$.
Write the equation in the form~(\ref{eq:one_var_eq_modified})
\begin{equation}\label{eq:main_th}
x^q = u(1) v(x).
\end{equation}

\noindent
Let $\alpha_{i,j} \in \mathbb{Q}$ be such that
\[
i < \alpha_{i,1} < \dots < \alpha_{i,q-1} < i+1, \qquad i=1,\dots,n-1,
\]
and let UT$_m(\mathbb{F}_p)$ be generated by
\[
t_{i, \alpha_{i,1}}', t_{\alpha_{i,1}, \alpha_{i,2}}', \dots, t_{\alpha_{i,q-1}, i+1}',
\quad i=1,\dots,n-1.
\]
By $\mathfrak{A}$ we denote the associative algebra, corresponding to UT$_m(\mathbb{F}_p)$, generated by the elements
\begin{equation}
\label{eq:algebra_generators}
e_{i, \alpha_{i,1}}', e_{\alpha_{i,1}, \alpha_{i,2}}', \dots, e_{\alpha_{i,q-1}, i+1}',
\quad i=1,\dots,n-1.
\end{equation}
Further we treat UT$_n(\mathbb{F}_p)$ as the subgroup of UT$_m(\mathbb{F}_p)$ with respect to embedding~(\ref{eq:UT_embedding_1}).

{\bf Case 1}.
Denote $G=$ UT$_n(\mathbb{F}_p)$ and $H=$ UT$_m(\mathbb{F}_p)$.
From properties of~(\ref{eq:UT_embedding_1}) it follows that $\gamma_i G \subseteq \gamma_{iq} H$, for $i=1,\dots,n-1$.
We will iteratively consider replicas of the equation~(\ref{eq:main_th}) in the factors $H / \gamma_i H$, for $i=2,\dots,m$.
During this process we will build the elements $x_i = 1 + u_i \in H$ such that the image of $x_i$ is a solution of the corresponding replica in $H / \gamma_i H$.
Finally we will obtain $x_m \in H$~--- a solution of~(\ref{eq:main_th}).
Observe that the cases $i=2,\dots,q$ could be skipped, since the images of $x_i=1$ are solutions of the corresponding replicas.

Solution of~(\ref{eq:main_th}) in the factor $H / \gamma_{q+1} H$  is the image of the following element
\begin{align*}
x_{q+1} = 1 &+ a_{1,2} e_{1,\alpha_{1,1}}' + e_{\alpha_{1,1},\alpha_{1,2}}' + \dots + e_{\alpha_{1,q-1},2}' \\
        &+ a_{2,3} e_{2,\alpha_{2,1}}' + e_{\alpha_{2,1},\alpha_{2,2}}' + \dots + e_{\alpha_{2,q-1},3}' \\
        &+ \dots \\
        &+ a_{n-1,n} e_{n-1,\alpha_{n-1,1}}' + e_{\alpha_{n-1,1},\alpha_{n-1,2}}' + \dots + e_{\alpha_{n-1,q-1},n}',
\end{align*}
since $v(x_{q+1}) \in \gamma_{q+1} H$.
Observe that $\Gamma(u_{q+1})$ contains the path from $\alpha_{1,1}$ to $n$, passing through the edges corresponding to the generators~(\ref{eq:algebra_generators}) of $\mathfrak{A}$.

Suppose that $x_i=1+u_i \in H$ is such that its image is a solution of the replica of~(\ref{eq:main_th}) in the factor $H / \gamma_i H$.
Then in the matrices $L(x_i)=x_i^q$ and $R(x_i)=u(1)v(x_i)$ first $i-1$ superdiagonals are identical.
Probably the matrices $L(x_i)$ and $R(x_i)$ are not equal modulo $\gamma_{i+1} H$ because of different coefficients of the elements $e_{\alpha,\beta} \in \mathfrak{A}^{(i)} \setminus \mathfrak{A}^{(i+1)}$, corresponding to the $i$-th superdiagonal.
There are $m-i$ such elements
\[
e_{\alpha_1,\beta_1}, e_{\alpha_2,\beta_2}, \dots, e_{\alpha_{m-i},\beta_{m-i}},
\]
where $\alpha_1 < \alpha_2 < \dots < \alpha_{m-i}$.
For each $e_{\alpha_j,\beta_j}$ ($j=m-i,\dots,1$) we will construct $y_j \in \mathfrak{A}$ such that the matrices $L(x_i + y_j + \dots + y_{m-i})$ and $R(x_i + y_j + \dots + y_{m-i})$ are equal modulo $\gamma_i H$ and the corresponding coefficients of $e_{\alpha_j,\beta_j}, e_{\alpha_{j+1},\beta_{j+1}}, \dots, e_{\alpha_{m-i},\beta_{m-i}}$ are equal.
Let $y_{j+1},\dots,y_{m-i}$ $(1 \leq j \leq m-i)$ be already constructed.
Write $w_l$, $w_r$ for the coefficients of $e_{\alpha_j,\beta_j}$ in $L(x_i + y_{j+1} + \dots + y_{m-i})$ and $R(x_i + y_{j+1} + \dots + y_{m-i})$ respectively.
The graph $\Gamma(u_i)$ contains the path $\mathcal{P}$ of length $q-1$ from some vertex $\tau$ (such a vertex is uniquely determined by $\beta_j$) to $\beta_j$, passing through the edges corresponding to the generators~(\ref{eq:algebra_generators}) of $\mathfrak{A}$.
Take $y_j = w_j e_{\alpha_j,\tau}$ where $w_j$ is defined by
\[
w_j w(\mathcal{P}) + w_l = w_r.
\]
According to Lemma~\ref{lm:power} and Remark~\ref{re:power} in the left part the element $w_j w(\mathcal{P}) e_{\alpha_j,\beta_j}$ will appear, and also some elements of the form $e_{\alpha,\beta} \in \mathfrak{A}^{(i+1)}$ or $e_{\alpha,\beta} \in \mathfrak{A}^{(i)} \setminus \mathfrak{A}^{(i+1)}$, furthermore in the latter case we will have $\beta < \beta_j$.
Observe that $e_{\alpha_j,\tau} \in \mathfrak{A}^{(i+1-q)}$, then, according to Lemma~\ref{lm:commutator},
\[
R(x_i + y_j + \dots + y_{m-i}) = R(x_i + y_{j+1} + \dots + y_{m-i}) + \mathcal{E},
\]
where $\mathcal{E} \in \mathfrak{A}^{(i+1)}$.
Thus addition of $y_j$ doesn't change the coefficients of $e_{\alpha_i,\beta_i}$ in the right part.
After addition of $y_1$ the left part remains unchanged.
Take $x_{i+1} = x_i + y_1 + \dots + y_{m-i}$.
By construction the image of $x_{i+1}$ is a solution of the replica of~(\ref{eq:main_th}) in $H / \gamma_{i+1} H$.

Continuing in this way we construct $x_m \in H$~--- a solution of~(\ref{eq:main_th}).

{\bf Case 2}.
In this case embedding~(\ref{eq:UT_embedding_2}) should be used instead of~(\ref{eq:UT_embedding_1}).
The remaining proof is similar to the case $1$.

{\bf Case 3}.
In this case we also treat UT$_n(\mathbb{F}_p)$ as the subgroup of UT$_m(\mathbb{F}_p)$ with respect to embedding~(\ref{eq:UT_embedding_1}).

Solution of the replica of~(\ref{eq:main_th}) in $H / \gamma_{q+1} H$ is the image of the element
\begin{align*}
x_{q+1} = 1 &+ e_{\alpha_{1,1},\alpha_{1,2}}' + \dots + e_{\alpha_{1,q-1},2}' \\
        &+ e_{\alpha_{2,1},\alpha_{2,2}}' + \dots + e_{\alpha_{2,q-1},3}' \\
        &+ \dots \\
        &+ e_{\alpha_{n-1,1},\alpha_{n-1,2}}' + \dots + e_{\alpha_{n-1,q-1},n}',
\end{align*}
since $x_{q+1}^q=1$ and $x_{q+1}$ commutes with the images of elements of $G$, i.e., $v(x_{q+1})=1$.
Further we continue as in the first case.
On iterations $q+2,\dots,m-1$ we will have $x_{q+2}=\dots=x_{m-1}=x_{q+1}$.
On iteration $m$ we obtain $x_m = x_{q+1} + a_{1,n} e_{1,\alpha_{n-1,1}}$~--- a solution of the equation.

The case of an arbitrary exponent now follows from Lemma~\ref{lm:prime_exponent}.

\end{proof}

From Theorem~\ref{th:main} we obtain the following result.

\begin{theorem}
\label{th:UT_3}
Any regular equation with exponent $rp^s$, where $r \in \mathbb{Z}$, $s \in \mathbb{N}$, $\gcd(r,p)=1$, over the Heisenberg $p$-group UT$_3(\mathbb{F}_p)$ is solvable in an overgroup isomorphic to UT$_{2p^s + 1}(\mathbb{F}_p)$.
\end{theorem}

\begin{proof}
It is clear that for any element of UT$_3(\mathbb{F}_p)$ one of the conditions of Theorem~\ref{th:main} holds.
\end{proof}

\end{document}